\documentclass[12pt]{amsart}
\usepackage{amsmath, amssymb, amsthm, mathtools, amsfonts, enumerate, verbatim}
\usepackage[all]{xy}
\usepackage{pgf,tikz}
\usepackage{mathrsfs}
\usepackage{subfig} 
\usepackage[margin=3.5mm,font={small},labelfont={normalsize}]{caption}
\usetikzlibrary{arrows}
\tikzstyle{punkt}=[circle, fill=black, minimum size=1mm,inner sep=0pt, draw]
\def\frk{\frak}               % font for "Fraktur"

\def\Phi{{\frk n}}
\def\Phi{{\frk N}}
\def\opn#1#2{\def#1{\operatorname{#2}}} % to make operators
\opn\chara{char}
\opn\length{\ell}
\opn\pd{pd}
\opn\rk{rk}
\opn\projdim{proj\,dim}
\opn\injdim{inj\,dim}
\opn\rank{rank}
\opn\depth{depth}
\opn\grade{grade}
\opn\height{height}
\opn\embdim{emb\,dim}
\opn\codim{codim}

\opn\Tr{Tr}
\opn\bigrank{big\,rank}
\opn\superheight{superheight}
\opn\lcm{lcm}
\opn\trdeg{tr\,deg}
\opn\reg{reg}
\opn\lreg{lreg}
\opn\ini{in}
\opn\lpd{lpd}
\opn\size{size}
\opn\bigsize{bigsize}
\opn\cosize{cosize}
\opn\bigcosize{bigcosize}
\opn\sdepth{sdepth}
\opn\sreg{sreg}
\opn\link{link}
\opn\fdepth{fdepth}
\opn\lin{lin}
\opn\ini{in}
\opn\div{div}
\opn\Div{Div}
\opn\cl{cl}
\opn\Cl{Cl}
\opn\Spec{Spec}
\opn\Supp{Supp}
\opn\supp{supp}
\opn\Sing{Sing}
\opn\Ass{Ass}
\opn\Min{Min}
\opn\Mon{Mon}
\opn\dstab{dstab}
\opn\astab{astab}
\opn\Syz{Syz}
\opn\Ann{Ann}
\opn\Rad{Rad}
\opn\Soc{Soc}
\opn\Im{Im}
\opn\Ker{Ker}
\opn\Coker{Coker}
\opn\Am{Am}
\opn\Hom{Hom}
\opn\Tor{Tor}
\opn\Ext{Ext}
\opn\End{End}
\opn\Aut{Aut}
\opn\id{id}

\opn\nat{nat}
\opn\pff{pf}%   \pf exists already
\opn\Pf{Pf}
\opn\GL{GL}
\opn\SL{SL}
\opn\mod{mod}
\opn\ord{ord}
\opn\Gin{Gin}
\opn\Hilb{Hilb}
\opn\sort{sort}
\opn\initial{init}
\opn\ende{end}
\opn\height{height}
\opn\type{type}
\opn\mdeg{mdeg}
\opn\aff{aff}
\opn\con{conv}
\opn\relint{relint}
\opn\st{st}
\opn\lk{lk}
\opn\cn{cn}
\opn\core{core}
\opn\vol{vol}
\opn\link{link}
\opn\star{star}
\opn\lex{lex}
\opn\sign{sign}
\opn\gr{gr}

\def\pot#1#2{#1[\kern-0.28ex[#2]\kern-0.28ex]}

\opn\dirlim{\underrightarrow{\lim}}
\opn\inivlim{\underleftarrow{\lim}}
\let\to=\rightarrow

\def\Implies{\ifmmode\Longrightarrow \else
	\unskip${}\Longrightarrow{}$\ignorespaces\fi}
\def\implies{\ifmmode\Rightarrow \else
	\unskip${}\Rightarrow{}$\ignorespaces\fi}
\def\iff{\ifmmode\Longleftrightarrow \else
	\unskip${}\Longleftrightarrow{}$\ignorespaces\fi}

\let\:=\colon
\newtheorem{Theorem}{Theorem}[section]
\newtheorem{Lemma}[Theorem]{Lemma}
\newtheorem{Corollary}[Theorem]{Corollary}
\newtheorem{Proposition}[Theorem]{Proposition}

\newtheorem{Example}[Theorem]{Example}

\newtheorem{Question}[Theorem]{Question}

\let\epsilon\varepsilon
\let\kappa=\varkappa
\def\pnt{{\raise0.5mm\hbox{\large\bf.}}}

\newcommand{\tg}{\triangleright}
\begin{document}
	\title{Complemented Lattices of Subracks}
	\author[A. Saki]{A. Saki}%$^{*}$
\address{Amir Saki, Department of Pure Mathematics, Faculty of Mathematics and Computer Science, Amirkabir University of Technology (Tehran Polytechnic), 424, Hafez Ave., Tehran 15914, Iran.}
\email{amir.saki.math@gmail.com}
\author[D. Kiani]{D. Kiani\,*}
\thanks{* Corresponding author}
\address{Dariush Kiani, Department of Pure Mathematics, Faculty of Mathematics and Computer Science, Amirkabir University of Technology (Tehran Polytechnic), 424, Hafez Ave., Tehran 15914, Iran, and School of Mathematics, Institute for Research in Fundamental Sciences (IPM), P.O. Box 19395-5746, Tehran, Iran.}
\email{dkiani@aut.ac.ir,  dkiani7@gmail.com}
\begin{abstract}
In this paper, a question due to Heckenberger, Shareshian and Welker on racks in \cite{Wel} is positively answered. A rack is a set together with a self-distributive bijective binary operation. We show that the lattice of subracks of every finite rack is complemented. Moreover, we characterize finite  modular lattices of subracks in terms of complements of subracks.  Also, we introduce a certain class of racks including all finite groups with the conjugation operation, called $G$-racks, and we study some of their properties. In particular, we show that a finite $G$-rack has the homotopy type of a sphere. Further, we show that the lattice of subracks of an infinite rack is not necessarily complemented which gives an affirmative answer to the aformentioned question. Indeed, we show that the lattice of subracks of the set of rational numbers, as a dihedral rack, is not complemented.  Finally, we show that being a Boolean algebra, pseudocomplemented and  uniquely complemented as well as distributivity are equivalent for the lattice of subracks of a rack. 

\smallskip
\vspace{5pt}
\end{abstract}
\maketitle

\section{Introduction}
Nowadays, racks  are used in several fields of pure and applied sciences  by knot theory, for example in  encoding the strands of knot diagrams. The study of an algebraic structure like a rack was started by M. Takasaki in \cite{Tak} where he studied reflections in  finite geometry. Indeed, he used a certain algebraic structure known as key or involutory quandle. Since then, racks, quandles and similar self-distributive systems were used in many papers for studying braids and knots (see \cite{Bri}, \cite{Fen}, \cite{Joy} and \cite{Mat}).  See also \cite{And}, for some  results about racks and their applications in Hopf algebras.

A \textit{rack} is a set together with a self-distributive and bijective operation. A \textit{subrack} of a rack $(R,\tg)$  is a subset $Q$ of $R$ where $(Q,\tg)$ is a rack. The set of all subracks of  a rack $R$  together with the inclusion order is a lattice. More precisely, the meet and the join of two subracks are the intersection of them and the subrack generated by them, respectively. The lattice of subracks of a rack was introduced by I. Heckenberger et al. in \cite{Wel} where the authors started the study of racks  from the combined perspective of both combinatorics and group theory, and they also posed some interesting questions about the lattice of subracks. In \cite{SK1},  a positive answer was given to one of those questions. Indeed, it was shown that the lattice of subracks of any rack is atomic. Moreover in \cite{SK1}, the atoms of the lattice of subracks were shown to be exactly all subracks generated by a singleton of $R$. 

In this paper, we  answer  the following question for finite racks, due to I. Heckenberger et al. in \cite{Wel}:
\begin{Question} \label{Que}
Is there a rack $R$ whose lattice of subracks  is not complemented? 
\end{Question}

Indeed, the above question is the first part of Question 5.2 in \cite{Wel}. In this paper, we give a positive answer to Question \ref{Que} by providing an example among infinite racks. However, we show that the lattice of subracks of any finite rack is complemented. 

Concerning  other properties of lattices of subracks one can mention \cite[Example~2.3]{Wel} where it was shown that the rack of all transpositions in the permutation group $S_n$ together with the conjugation operation is isomorphic to the lattice of all partitions of a set with $n$ elements. It was also shown in \cite[Corollary 2.11]{Wel}  that the topology of  the lattice of subracks of a rack, in the sense of its order complex, could be complicated. 

This paper is organized as follows. In Section \ref{sec}, we provide some preliminaries which we use throughout the paper.  Section \ref{sec1} is devoted to Question \ref{Que} in the case of finite racks and some related problems arising from this question. Indeed, we show that the lattice of subracks of any finite rack is complemented.  Moreover, we show  how  modularity and  complements of subracks in the lattice of subracks of a finite rack are related. In Section \ref{G}, we  introduce a certain class of racks including all finite groups with the conjugation operation, called  $G$-racks, and we study some  of their properties such as their homotopy types. In Section \ref{S}, we eventually give a positive answer to Question \ref{Que} using infinite racks. Actually, we show that the abelian group $\mathbb{Q}$, as a dihedral rack, is not complemented. At the end of this section, we show that being a Boolean algebra, pseudocomplemented, uniquely complemented and distributivity are equivalent for the lattice of subracks of a rack.

\section{Preliminaries}\label{sec}
A lattice is a  partially ordered set  in which every two elements have a unique least upper bound (supremum) and a unique greatest lower bound (infimum). Let $L$ be a lattice. We call the supremum and the infimum of the two elements $a,b\in L$, the \textit{join} and the \textit{meet} of $a$ and $b$, respectively, and denote them by $a\vee b$ and $a\wedge b$, respectively. If $L$ has the greatest (resp. least) element, then this element is denoted by $\hat{1}$ (resp. $\hat{0}$). A lattice including $\hat{0}$ and $\hat{1}$ is called \textit{bounded}. The lattice $L$ is called \textit{distributive} if $a\wedge(b\vee c)=(a\wedge b)\vee(a\wedge c)$, for all $a,b,c\in L$. Modularity  is a simple generalization of  distributivity. Indeed, the lattice $L$ is called \textit{modular} if $a\wedge(c\vee b)=(a\wedge c)\vee b$, for all $a,b,c\in L$ with $b\le a$. The lattice $L$ including $\hat{0}$ and $\hat{1}$ is called \textit{complemented} if for every $a\in L$ there exists an element $b\in L$, called the complement of $a$, such that $a\wedge b=\hat{0}$ and $a\vee b=\hat{1}$. If in a complemented lattice every element has a unique complement, then the lattice is called \textit{uniquely complemented}. A complemented distributive lattice is called a \textit{Boolean algebra}. We denote  the lattice obtained by removing $\hat{0}$ and $\hat{1}$ (if  exist) by $\overline{L}$. An \textit{atom} of $L$  is an element $a\in L$ such that $b\le a$ implies  $b=a$, for every $b\in\overline{L}$. A lattice whose any of elements is the join of some atoms is called \textit{atomic}. A \textit{relatively atomic} lattice is a lattice in which every interval is atomic. Similarly, a \textit{relatively complemented} lattice is a lattice in which every closed bounded interval $[x,y]$  is complemented. It is obvious that every bounded relatively atomic lattice (resp. bounded relatively complemented lattice) is atomic (resp. complemented). The lattice $L$ including $\hat{0}$ is called \textit{pseudocomplemented} if for every $x\in L$ there exists a greatest element $x^*$ with  $x\wedge x^*=\hat{0}$ (here $x^*$ is unique).

A (finite abstract) \textit{simplicial complex} is a collection $\Delta$ of sets  such that for each $\sigma\in\Delta$ and $\tau\subseteq\sigma$, one has $\tau\in\Delta$. Let $\Delta$ be a simplicial complex. The elements of $\Delta$ are called \textit{faces} or \textit{simplices} of $\Delta$.  Maximal faces of $\Delta$ are called  \textit{facets}. The dimension of a face of $\Delta$ is defined to be the cardinality of the face minus one. The dimension of $\Delta$ is the maximum of the dimensions of its faces. A subcomplex of $\Delta$ is a subset $\Gamma$ of $\Delta$ which is a simplicial complex. 

A  rack  is a set $R$ together with a binary operation $\tg$ such that the following conditions are satisfied: for all $a,b,c\in R$, we have $a\tg(b\tg c)=(a\tg b)\tg(a\tg c)$, and there exists a unique $x\in R$ with $a\tg x=b$. A \textit{quandle} is a rack $Q$ with the additional condition $x\tg x=x$, for all $x\in Q$.

Let $(R,\tg)$ be a rack. A  subrack  of $R$ is a subset $Q$ of $R$ where $(Q,\tg)$ is a rack as well. The set of all subracks of $R$ with the inclusion order is a lattice called the\textit{ lattice of subracks} of $R$, and denoted by $\mathcal{R}(R)$. For a subset $S$ of $R$, the \textit{subrack generated} by $S$ in $R$, denoted by $\ll S\gg$, is the intersection of all subracks of $R$ including $S$. It follows from the definition that  for every $a\in R$, the map $f_a:R\to R$ with $f_a(b)=a\tg b$ is a bijection. An automorphism of $R$ is a function $\phi:R\to R$ such that for all $a,b\in R$, we have $\phi(a\tg b)=\phi(a)\tg\phi(b)$. The \textit{inner group} of $R$, denoted by $\mathrm{Inn}(R)$, is the subgroup generated by $\{f_a:a\in R\}$ in the automorphism group of $R$. Note that $f_a$ is an automorphism of $R$, for every $a\in R$. The inner group of $R$ acts on $R$ by the natural action $\phi*x=\phi(x)$, for all $\phi\in\mathrm{Inn}(R)$ and $x\in R$. The orbits of this action are subracks of $R$. In other words, the orbit including $x\in R$ is the set $\{\phi(x): \phi\in\mathrm{Inn}(R)\}$. Note that when we refer to a group as a rack, we mean that the operation is conjugation.

\section{Complementedness of The Lattice of Subracks of Finite Racks} \label{sec1}

In this section, we discuss Question \ref{Que} for finite racks. Indeed, we show that $\mathcal{R}(R)$ is complemented for every finite rack $R$. We also determine some other structural properties of $\mathcal{R}(R)$. In particular, we give a characterization of finite modular lattices of subracks in terms of complements of the elements  of the underlying racks. The study of Question \ref{Que} in the case of infinite racks is postponed to Section \ref{S} where a positive answer to this question is provided.

\medskip 
The following theorem is one of the most important results of this paper which immediately concludes that the lattice of subracks of a finite rack is complemented.

\begin{Theorem} \label{Thm1}
Let $R$ be a finite rack. Then the following statements hold:
\begin{enumerate}
\item
The intersection of all maximal subracks of $R$ is empty.
\item
For any two subracks $Q_1$ and $Q_2$ of $R$ with $T=\ll Q_1, Q_2\gg$, the subrack $Q_1$ has a complement in $\mathcal{R}(T)$ which is a subset of $Q_2$. 
\end{enumerate}
\end{Theorem}

\begin{proof}
(1) If $R=\emptyset$, then the result is clear. Now, we  assume that $R\neq\emptyset$. Suppose that $\{M_i\}_{i=1}^n$ is the set of all maximal subracks of $R$ and $X=\bigcap_{i=1}^nM_i$. Let $\phi$ be an  automorphism of $R$. It is easily seen  that the map $M_i\mapsto \phi(M_i)$ is a permutation on the set of all maximal subracks of $R$. Therefore, we have 
\[ \phi(X)=\phi\left(\bigcap_{i=1}^nM_i\right)=\bigcap_{i=1}^n\phi(M_i)=\bigcap_{i=1}^nM_i=X,\]
which implies that $X$ is the union of a set of orbits of $R$. Note that $R\backslash X$ is the union of the other orbits, and hence it is a subrack. If $X\neq\emptyset$, then there is a maximal subrack $M$ containing $R\backslash X$. We also have $X\subseteq M$ and hence $M=R$ which is a contradiction.

(2) Let
\[ \mathcal{C}\coloneqq\{|Q_1\cap Q|: \text{$Q$ is a subrack of $Q_2$ and $T=\ll Q_1, Q\gg$}\}.\]
Obviously, $|Q_1\cap Q_2|$ belongs to $\mathcal{C}$, so that $\mathcal{C}$ is nonempty. Assume that $|Q_1\cap Q_3|$ is the minimum element of $\mathcal{C}$ for a subrack $Q_3$ of $Q_2$. We show that $Q_3$ is a complement of $Q_1$ in $\mathcal{R}(T)$. To do this, it is enough to show that  $Q_1\cap Q_3=\emptyset$. Suppose on contrast that  $x\in Q_1\cap Q_3$. By the part (1) the intersection of all maximal subracks of $Q_3$ is  empty, so there exists a maximal subrack $Q_4$ of $Q_3$ which does not contain $x$. Thus by Theorem 2.3 of \cite{SK1}, $Q_4$ is a complement of $\ll x\gg$ in $\mathcal{R}(Q_3)$. Now, we have $|Q_1\cap Q_4|<|Q_1\cap Q_3|$, while 
\[\ll Q_1, Q_4\gg= \ll Q_1, x, Q_4\gg=\ll Q_1, \ll x, Q_4\gg\gg=\ll Q_1, Q_3\gg= T.\] 
But this is  a contradiction, since $|Q_1\cap Q_3|$ is the minimum of $\mathcal{C}$. 
\end{proof}
\begin{Corollary}
The lattice of subracks of every finite rack is complemented.
\end{Corollary}

\begin{proof}
Let $R$ be a rack and $Q$ be a subrack of $R$. Then we have $\ll Q, R\gg=R$, and hence by the part (2) of Theorem \ref{Thm1}, the result follows.
\end{proof}

There is an important relationship between the lattices of subracks of racks and quandles. The \textit{corresponding quandle} of $R$, denoted by $(\bar{R},*)$, is the set of all atoms of $R$ together with the following operation: for any two atoms $A$ and $B$ of $R$ with $a\in A$ and $b\in B$, we have $A*B=C$ where $C$ is the atom which includes $a\tg b$. We may sometimes consider  quandles instead of racks, because the lattice of subracks of a rack and its corresponding quandle are isomorphic, by  \cite[Corollary~2.7]{SK1}.
 
%In the next theorem, we provide some equivalent statements concerning the structure of the lattice of subracks of a rack.

Note that the lattice of subracks of a rack is not necessarily relatively complemented. For  example, let $Q$ be the quandle $S_3$.   Since every subrack of $Q$ including a 2-cycle and a 3-cycle includes all other 2-cycles and 3-cycles, it follows that the interval $\left[(1\quad 2), Q\right]$ is not complemented (see Figure \ref{S3}).

Similarly, the lattice of subracks of a rack is not necessarily relatively atomic. For  example, let $Q$ be the dihedral quandle  $\mathbb{Z}_{8}$ (i.e. $x\tg y=2x-y$ for all $x,y\in\mathbb{Z}_8$). In this quandle, the interval $\left[\{0\}, Q\right]$ is not atomic. Indeed, this interval has only the atom $\{0,4\}$, and hence the subrack $\{0,2,4,6\}$ is not the join of some atoms of this interval. Furthermore, this interval  is not complemented, because every subrack of $Q$ which includes an odd number and an even number must generate  $Q$ (see Figure \ref{Z8}). 

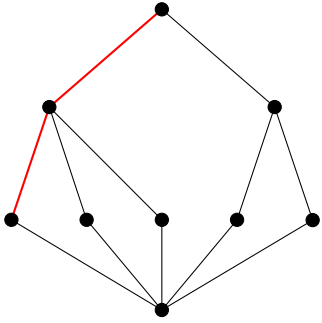
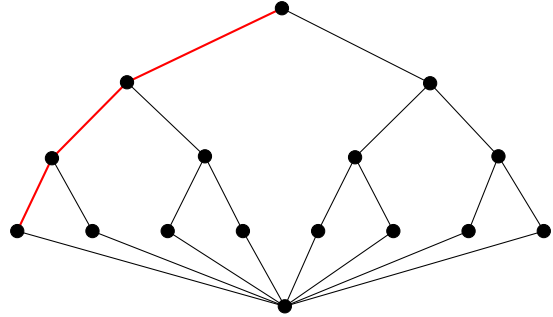
\begin{figure}[h]
\centering
\subfloat[The lattice~$\mathcal{R}(S_3\backslash\{id\})$, and a non-complemented atomic interval of it (shown in red).\label{S3}]{
\definecolor{ffqqqq}{rgb}{1.,0.,0.}
\begin{tikzpicture}[line cap=round,line join=round,>=triangle 45,x=1.0cm,y=1.0cm]
\clip(4.9,-2.2998147694362574) rectangle (9.438965103515084,2.257682708337581);
\draw [line width=0.4pt] (7.,-2.)-- (5.,-0.8);
\draw [line width=0.4pt] (7.,-2.)-- (6.,-0.8);
\draw [line width=0.4pt] (7.,-2.)-- (7.,-0.8);
\draw [line width=0.4pt] (7.,-2.)-- (8.,-0.8);
\draw [line width=0.4pt] (7.,-2.)-- (9.004851665309076,-0.8040825942618293);
\draw [line width=0.4pt] (7.,-0.8)-- (5.50419,0.7);
\draw [line width=0.4pt] (6.,-0.8)-- (5.50419,0.7);
\draw [line width=0.8pt,color=ffqqqq] (5.,-0.8)-- (5.50419,0.7);
\draw [line width=0.4pt] (8.,-0.8)-- (8.5,0.7);
\draw [line width=0.4pt] (9.004851665309076,-0.8040825942618293)-- (8.5,0.7);
\draw [line width=0.8pt,color=ffqqqq] (5.50419,0.7)-- (7.,2.);
\draw [line width=0.4pt] (8.5,0.7)-- (7.,2.);
\draw [line width=0.4pt] (13.55718533921923,-2.0012363055269073)-- (10.,-1.);
\draw [line width=0.4pt] (11.,-1.)-- (13.55718533921923,-2.0012363055269073);
\draw [line width=0.4pt] (13.55718533921923,-2.0012363055269073)-- (12.,-1.);
\draw [line width=0.4pt] (13.,-1.)-- (13.55718533921923,-2.0012363055269073);
\draw [line width=0.4pt] (13.55718533921923,-2.0012363055269073)-- (14.,-1.);
\draw [line width=0.4pt] (13.55718533921923,-2.0012363055269073)-- (15.,-1.);
\draw [line width=0.4pt] (13.55718533921923,-2.0012363055269073)-- (16.,-1.);
\draw [line width=0.4pt] (13.55718533921923,-2.0012363055269073)-- (17.,-1.);
\draw [line width=0.8pt,color=ffqqqq] (10.,-1.)-- (10.461236731966114,-0.031087191820382755);
\draw [line width=0.4pt] (11.,-1.)-- (10.461236731966114,-0.031087191820382755);
\draw [line width=0.4pt] (12.,-1.)-- (12.495351725987788,-0.005500839694323996);
\draw [line width=0.4pt] (12.495351725987788,-0.005500839694323996)-- (13.,-1.);
\draw [line width=0.4pt] (14.,-1.)-- (14.491087191820375,-0.018294015757353375);
\draw [line width=0.4pt] (14.491087191820375,-0.018294015757353375)-- (15.,-1.);
\draw [line width=0.4pt] (16.,-1.)-- (16.397270425211754,-0.005500839694323996);
\draw [line width=0.4pt] (16.397270425211754,-0.005500839694323996)-- (17.,-1.);
\draw [line width=0.8pt,color=ffqqqq] (10.461236731966114,-0.031087191820382755)-- (11.459104464882406,0.9795737171589383);
\draw [line width=0.4pt] (11.459104464882406,0.9795737171589383)-- (12.495351725987788,-0.005500839694323996);
\draw [line width=0.4pt] (14.491087191820375,-0.018294015757353375)-- (15.48895492473667,0.9667805410959089);
\draw [line width=0.4pt] (15.48895492473667,0.9667805410959089)-- (16.397270425211754,-0.005500839694323996);
\draw [line width=0.8pt,color=ffqqqq] (11.459104464882406,0.9795737171589383)-- (13.518805811030141,1.9646482740122007);
\draw [line width=0.4pt] (15.48895492473667,0.9667805410959089)-- (13.518805811030141,1.9646482740122007);
\begin{scriptsize}
\draw [fill=black] (7.,-2.) circle (2.5pt);
\draw [fill=black] (7.,-0.8) circle (2.5pt);
\draw [fill=black] (6.,-0.8) circle (2.5pt);
\draw [fill=black] (5.,-0.8) circle (2.5pt);
\draw [fill=black] (8.,-0.8) circle (2.5pt);
\draw [fill=black] (9.004851665309076,-0.8040825942618293) circle (2.5pt);
\draw [fill=black] (5.50419,0.7) circle (2.5pt);
\draw [fill=black] (8.5,0.7) circle (2.5pt);
\draw [fill=black] (7.,2.) circle (2.5pt);
\draw [fill=black] (10.,-1.) circle (2.5pt);
\draw [fill=black] (11.,-1.) circle (2.5pt);
\draw [fill=black] (12.,-1.) circle (2.5pt);
\draw [fill=black] (13.,-1.) circle (2.5pt);
\draw [fill=black] (14.,-1.) circle (2.5pt);
\draw [fill=black] (15.,-1.) circle (2.5pt);
\draw [fill=black] (16.,-1.) circle (2.5pt);
\draw [fill=black] (17.,-1.) circle (2.5pt);
\draw [fill=black] (13.55718533921923,-2.0012363055269073) circle (2.5pt);
\draw [fill=black] (10.461236731966114,-0.031087191820382755) circle (2.5pt);
\draw [fill=black] (12.495351725987788,-0.005500839694323996) circle (2.5pt);
\draw [fill=black] (14.491087191820375,-0.018294015757353375) circle (2.5pt);
\draw [fill=black] (16.397270425211754,-0.005500839694323996) circle (2.5pt);
\draw [fill=black] (11.459104464882406,0.9795737171589383) circle (2.5pt);
\draw [fill=black] (15.48895492473667,0.9667805410959089) circle (2.5pt);
\draw [fill=black] (13.518805811030141,1.9646482740122007) circle (2.5pt);
\end{scriptsize}
\end{tikzpicture}
}
\qquad
\subfloat[The lattice of subracks of $\mathbb{Z}_8$ as a dihedral rack, and an interval of it (shown in red) which is neither complemented nor atomic.\label{Z8}]{
\definecolor{ffqqqq}{rgb}{1.,0.,0.}
\begin{tikzpicture}[line cap=round,line join=round,>=triangle 45,x=1.0cm,y=1.0cm]
\clip(9.747996143653158,-2.2707398253994224) rectangle (17.456380921672192,2.2867576523744155);
\draw [line width=0.4pt] (7.,-2.)-- (5.,-0.8);
\draw [line width=0.4pt] (7.,-2.)-- (6.,-0.8);
\draw [line width=0.4pt] (7.,-2.)-- (7.,-0.8);
\draw [line width=0.4pt] (7.,-2.)-- (8.,-0.8);
\draw [line width=0.4pt] (7.,-2.)-- (9.004851665309076,-0.8040825942618293);
\draw [line width=0.4pt] (7.,-0.8)-- (5.50419,0.7);
\draw [line width=0.8pt,color=ffqqqq] (6.,-0.8)-- (5.50419,0.7);
\draw [line width=0.4pt] (5.,-0.8)-- (5.50419,0.7);
\draw [line width=0.4pt] (8.,-0.8)-- (8.5,0.7);
\draw [line width=0.4pt] (9.004851665309076,-0.8040825942618293)-- (8.5,0.7);
\draw [line width=0.8pt,color=ffqqqq] (5.50419,0.7)-- (7.,2.);
\draw [line width=0.4pt] (8.5,0.7)-- (7.,2.);
\draw [line width=0.4pt] (13.55718533921923,-2.0012363055269073)-- (10.,-1.);
\draw [line width=0.4pt] (11.,-1.)-- (13.55718533921923,-2.0012363055269073);
\draw [line width=0.4pt] (13.55718533921923,-2.0012363055269073)-- (12.,-1.);
\draw [line width=0.4pt] (13.,-1.)-- (13.55718533921923,-2.0012363055269073);
\draw [line width=0.4pt] (13.55718533921923,-2.0012363055269073)-- (14.,-1.);
\draw [line width=0.4pt] (13.55718533921923,-2.0012363055269073)-- (15.,-1.);
\draw [line width=0.4pt] (13.55718533921923,-2.0012363055269073)-- (16.,-1.);
\draw [line width=0.4pt] (13.55718533921923,-2.0012363055269073)-- (17.,-1.);
\draw [line width=0.8pt,color=ffqqqq] (10.,-1.)-- (10.461236731966114,-0.031087191820382755);
\draw [line width=0.4pt] (11.,-1.)-- (10.461236731966114,-0.031087191820382755);
\draw [line width=0.4pt] (12.,-1.)-- (12.495351725987788,-0.005500839694323996);
\draw [line width=0.4pt] (12.495351725987788,-0.005500839694323996)-- (13.,-1.);
\draw [line width=0.4pt] (14.,-1.)-- (14.491087191820375,-0.018294015757353375);
\draw [line width=0.4pt] (14.491087191820375,-0.018294015757353375)-- (15.,-1.);
\draw [line width=0.4pt] (16.,-1.)-- (16.397270425211754,-0.005500839694323996);
\draw [line width=0.4pt] (16.397270425211754,-0.005500839694323996)-- (17.,-1.);
\draw [line width=0.8pt,color=ffqqqq] (10.461236731966114,-0.031087191820382755)-- (11.459104464882406,0.9795737171589383);
\draw [line width=0.4pt] (11.459104464882406,0.9795737171589383)-- (12.495351725987788,-0.005500839694323996);
\draw [line width=0.4pt] (14.491087191820375,-0.018294015757353375)-- (15.48895492473667,0.9667805410959089);
\draw [line width=0.4pt] (15.48895492473667,0.9667805410959089)-- (16.397270425211754,-0.005500839694323996);
\draw [line width=0.8pt,color=ffqqqq] (11.459104464882406,0.9795737171589383)-- (13.518805811030141,1.9646482740122007);
\draw [line width=0.4pt] (15.48895492473667,0.9667805410959089)-- (13.518805811030141,1.9646482740122007);
\begin{scriptsize}
\draw [fill=black] (7.,-2.) circle (2.5pt);
\draw [fill=black] (7.,-0.8) circle (2.5pt);
\draw [fill=black] (6.,-0.8) circle (2.5pt);
\draw [fill=black] (5.,-0.8) circle (2.5pt);
\draw [fill=black] (8.,-0.8) circle (2.5pt);
\draw [fill=black] (9.004851665309076,-0.8040825942618293) circle (2.5pt);
\draw [fill=black] (5.50419,0.7) circle (2.5pt);
\draw [fill=black] (8.5,0.7) circle (2.5pt);
\draw [fill=black] (7.,2.) circle (2.5pt);
\draw [fill=black] (10.,-1.) circle (2.5pt);
\draw [fill=black] (11.,-1.) circle (2.5pt);
\draw [fill=black] (12.,-1.) circle (2.5pt);
\draw [fill=black] (13.,-1.) circle (2.5pt);
\draw [fill=black] (14.,-1.) circle (2.5pt);
\draw [fill=black] (15.,-1.) circle (2.5pt);
\draw [fill=black] (16.,-1.) circle (2.5pt);
\draw [fill=black] (17.,-1.) circle (2.5pt);
\draw [fill=black] (13.55718533921923,-2.0012363055269073) circle (2.5pt);
\draw [fill=black] (10.461236731966114,-0.031087191820382755) circle (2.5pt);
\draw [fill=black] (12.495351725987788,-0.005500839694323996) circle (2.5pt);
\draw [fill=black] (14.491087191820375,-0.018294015757353375) circle (2.5pt);
\draw [fill=black] (16.397270425211754,-0.005500839694323996) circle (2.5pt);
\draw [fill=black] (11.459104464882406,0.9795737171589383) circle (2.5pt);
\draw [fill=black] (15.48895492473667,0.9667805410959089) circle (2.5pt);
\draw [fill=black] (13.518805811030141,1.9646482740122007) circle (2.5pt);
\end{scriptsize}
\end{tikzpicture}
}
\caption{The lattice of subracks of a rack is not necessarily  relatively complemented or relatively atomic.}
\end{figure}

It is well-known that every complemented modular lattice is relatively complemented. To see this, let $L$ be a complemented modular lattice and $[a,b]$ be an interval in $L$. For an element $c\in[a,b]$, if $c'$ is a complement of $c$ in $L$, then $a\vee(c'\wedge b)$ is a complement of $c$ in $[a,b]$, because by modularity we have $c\wedge\left(a\vee(c'\wedge b)\right)=\left(c\wedge(c'\wedge b)\right)\vee a=a$, and we also have $c\vee\left(a\vee(c'\wedge b)\right)=c\vee(c'\wedge b)=b\wedge(c\vee c')=b$.

A \textit{sublattice} $S$ of $L$ is a subset of $L$ such that the join and the meet of every two elements of $S$ in $L$ belong to $S$. Note that it is a well-known fact that a lattice $L$ is modular if and only if it does not have any sublattice isomorphic to  the lattice in Figure \ref{f1} (see \cite[Theorem 2 of Section 7]{Gra}). As some well-known examples of modular lattices, we refer to the lattice of submodules of a module over a ring, and the lattice of normal subgroups of a group. Since the lattice of subracks is atomic, modular lattice of subracks are geometric lattices. The following theorem provides a relationship between modularity and reducing complements in the lattice of subracks of a finite rack.

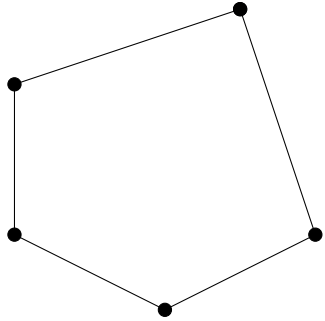
\begin{figure}
\definecolor{ududff}{rgb}{0.30196078431372547,0.30196078431372547,1.}
\begin{tikzpicture}[line cap=round,line join=round,>=triangle 45,x=1.0cm,y=1.0cm]
\clip(0,-2.1747111242953716) rectangle (7.837052219488011,2.22048237520128);
\draw [line width=0.4pt] (4.,-2.)-- (2.,-1.);
\draw [line width=0.4pt] (4.,-2.)-- (6.,-1.);
\draw [line width=0.4pt] (2.,-1.)-- (2.,1.);
\draw [line width=0.4pt] (2.,1.)-- (5.,2.);
\draw [line width=0.4pt] (6.,-1.)-- (5.,2.);
\begin{scriptsize}
\draw [fill=black] (4.,-2.) circle (2.5pt);
\draw [fill=black] (2.,-1.) circle (2.5pt);
\draw [fill=black] (6.,-1.) circle (2.5pt);
\draw [fill=black] (2.,1.) circle (2.5pt);
\draw [fill=ududff] (5.,2.) circle (2.5pt);
\draw [fill=black] (5.,2.) circle (2.5pt);
\end{scriptsize}
\end{tikzpicture}
\caption{The lattice $N_5$ as  the smallest lattice which is not modular.\label{f1}}
\end{figure}

Let $R$ be a finite rack and let $Q$, $R_1$ and $R_2$ be subracks of $R$ with $Q\subseteq R_1\subseteq R_2$. Then we say that a complement $Q'$ of $Q$ in $\mathcal{R}(R_2)$ can be reduced to a complement of $Q$ in $\mathcal{R}(R_1)$, if $Q'\cap R_1$ is a complement of $Q$ in $\mathcal{R}(R_1)$.
\begin{Theorem}\label{Thm2}
Let $R$ be a finite rack. Then $\mathcal{R}(R)$ is modular if and only if for all subracks $Q,R_1$ and $R_2$ of $R$ with $Q\subseteq R_1\subseteq R_2$  every complement of $Q$ in $\mathcal{R}(R_2)$ can be reduced to a complement of $Q$ in $\mathcal{R}(R_1)$. 
\end{Theorem}

\begin{proof}
First, suppose that $\mathcal{R}(R)$ is modular, and $Q,R_1$ and $R_2$ are subracks of $R$ with $Q\subseteq R_1\subseteq R_2$. Let $Q'$ be a complement of $Q$ in $\mathcal{R}(R_2)$. By modularity, we have $R_1\cap\ll Q',Q\gg=\ll R_1\cap Q',Q\gg$. It follows that $R_1=\ll R_1\cap Q',Q\gg$, and hence $R_1\cap Q'$ is a complement of $Q$ in $\mathcal{R}(R_1)$. 

Conversely, suppose that we can reduce every complement of $Q$ in $\mathcal{R}(R_2)$ to a complement of $Q$ in $\mathcal{R}(R_1)$ for every subrack $Q$ of $R_2$. If $\mathcal{R}(R)$ is not modular, then $\mathcal{R}(R)$ has  a  sublattice   as shown in  Figure \ref{NM}.
\begin{figure}
\centering
\subfloat[The lattice $N_5$ in a non-modular lattice of subracks.\label{NM}]{
\definecolor{ududff}{rgb}{0.30196078431372547,0.30196078431372547,1.}
\begin{tikzpicture}[line cap=round,line join=round,>=triangle 45,x=1.0cm,y=1.0cm]
\clip(1.3,-2.55) rectangle (7,3);
\draw [line width=0.4pt] (4.,-2.)-- (2.,-1.);
\draw [line width=0.4pt] (4.,-2.)-- (6.,-1.);
\draw [line width=0.4pt] (2.,-1.)-- (2.,1.);
\draw [line width=0.4pt] (2.,1.)-- (5.,2.);
\draw [line width=0.4pt] (6.,-1.)-- (5.,2.);
\draw (3.15,-2.0354248486414646) node[anchor=north west] {$Q_1\cap Q_2$};
\draw (1.2,-0.6) node[anchor=north west] {$Q_1$};
\draw (6.126431693679114,-0.6) node[anchor=north west] {$Q_2$};
\draw (1.2,1.5) node[anchor=north west] {$Q_3$};
\draw (3.75,2.8) node[anchor=north west] {$\ll Q_1, Q_2\gg$};
\begin{scriptsize}
\draw [fill=black] (4.,-2.) circle (2.5pt);
\draw [fill=black] (2.,-1.) circle (2.5pt);
\draw [fill=black] (6.,-1.) circle (2.5pt);
\draw [fill=black] (2.,1.) circle (2.5pt);
\draw [fill=ududff] (5.,2.) circle (2.5pt);
\draw [fill=black] (5.,2.) circle (2.5pt);
\end{scriptsize}
\end{tikzpicture}
}
\qquad
\subfloat[The lattice $N_5$  which has an empty bottom element.\label{NM1}]{
\definecolor{ududff}{rgb}{0.30196078431372547,0.30196078431372547,1.}
\begin{tikzpicture}[line cap=round,line join=round,>=triangle 45,x=1.0cm,y=1.0cm]
\clip(1.37,-2.55) rectangle (7,3);
\draw [line width=0.4pt] (4.,-2.)-- (2.,-1.);
\draw [line width=0.4pt] (4.,-2.)-- (6.,-1.);
\draw [line width=0.4pt] (2.,-1.)-- (2.,1.);
\draw [line width=0.4pt] (2.,1.)-- (5.,2.);
\draw [line width=0.4pt] (6.,-1.)-- (5.,2.);
\draw (3.76,-2.0354248486414646) node[anchor=north west] {$\emptyset$};
\draw (1.2,-0.6) node[anchor=north west] {$Q_1$};
\draw (6.126431693679114,-0.6) node[anchor=north west] {$Q_2'$};
\draw (1.2,1.5) node[anchor=north west] {$Q_3$};
\draw (3.75,2.8) node[anchor=north west] {$\ll Q_1, Q_2'\gg$};
\begin{scriptsize}
\draw [fill=black] (4.,-2.) circle (2.5pt);
\draw [fill=black] (2.,-1.) circle (2.5pt);
\draw [fill=black] (6.,-1.) circle (2.5pt);
\draw [fill=black] (2.,1.) circle (2.5pt);
\draw [fill=ududff] (5.,2.) circle (2.5pt);
\draw [fill=black] (5.,2.) circle (2.5pt);
\end{scriptsize}
\end{tikzpicture}
}
\caption{One can construct the right  sublattice from the left one.}
\end{figure}
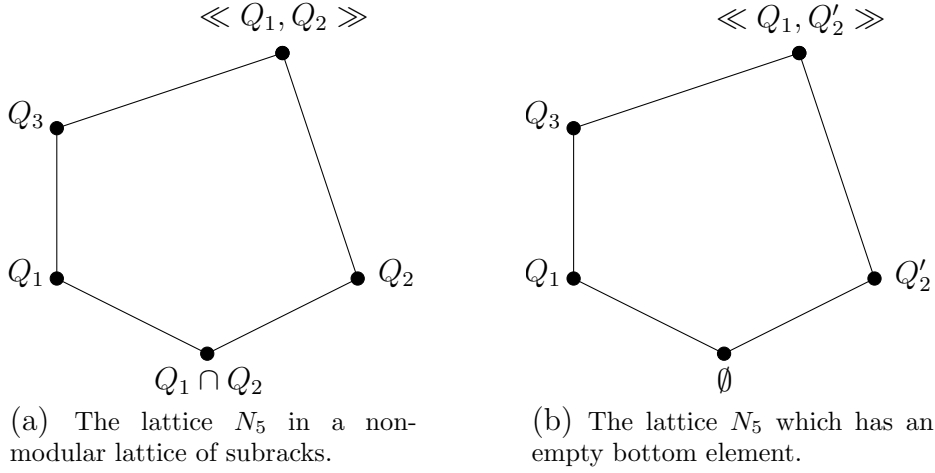
It follows from the part (2) of Theorem \ref{Thm1} that there exists a subrack $Q'_2$ of $Q_2$ which is a complement of $Q_1$ in $\mathcal{R}(T)$ with $T=\ll Q_1,Q_2\gg$, and hence $Q'_2\cap Q_1=\emptyset$ and $T=\ll Q_1, Q'_2\gg$. We have $Q'_2\cap Q_3\subseteq Q_2\cap Q_3=Q_2\cap Q_1$, and hence $Q'_2\cap Q_3=\emptyset$ because the intersection of $Q_1$ and $Q'_2$ is empty. We therefore obtain the  sublattice of $\mathcal{R}(R)$ shown in Figure~\ref{NM1}. By our assumption, we can reduce $Q'_2$ to  a complement of $Q_1$ in $\mathcal{R}(Q_3)$, namely $Q'_2\cap Q_3$ is a complement of $Q_1$ in $\mathcal{R}(Q_3)$. But this is impossible because $Q'_2\cap Q_3$ is  empty, a contradiction. Therefore $\mathcal{R}(R)$ is a modular lattice.
\end{proof}

\section{$G$-Racks}\label{G}

In this section, we introduce a certain class of racks called  \textit{$G$-racks}, (here "$G$" stands for  "Group", motivated by the fact that this class of racks includes all finite groups). We say that a rack $R$ is a $G$-rack if $R$ is the only subrack of $R$ which has a nonempty intersection with every orbit of $R$. By  \cite[Lemma 2.8]{Wel}, all finite groups are $G$-racks. Note that the orbits of a group  are its conjugacy classes. An example of a rack whose lattice of subracks is not distributive was introduced in \cite{SK1}. We call such a rack a \textit{$P$-rack}, and then  we show that $\mathrm{Inn}(R)$ is an abelian group if and only if $R$ is a $P$-rack for all racks $R$. Next, we give an example of an infinite group which is not a $G$-rack. Finally, we provide some equivalent properties to being a Boolean algebra for $G$-racks.

\medskip 
Now, we introduce $P$-racks. Let $R$ be a set, and  $\{R_i\}_{i\in I}$  be a partition of $R$. Assume that $\{f_i\}_{i\in I}$ is a family of permutations on $R$ such that $f_i(R_j)=R_j$ and $f_if_j=f_jf_i$ for all $i,j\in I$. If $x\tg y=f_i(y)$  for all $x,y\in R$ where $x\in R_i$, then $(R,\tg)$ is a rack  which we call a $P$-rack (here "P" comes from "Partition"). This rack is a quandle if and only if for each $i\in I$ the restriction of $f_i$ to $R_i$ is the identity. We denote such a $P$-rack by $\left(R,\{R_i\}_{i\in I},\{f_i\}_{i\in I}\right)$. See also \cite[Example 2.9]{SK1}.

In the next proposition, we show that any $P$-rack is a  $G$-rack.

\begin{Proposition}
All $P$-racks are $G$-racks.
\end{Proposition}

\begin{proof}
Let $\left(R,\{R_i\}_{i\in I},\{f_i\}_{i\in I}\right)$ be a $P$-rack. Suppose that $\mathcal{O}$ is the collection of all orbits of $R$, and $Q$ is a subrack of $R$ such that $Q\cap X\neq\emptyset$ for every $X\in\mathcal{O}$. Note that for each orbit $X$ of $R$, there is a unique $i\in I$ for which $X\subseteq R_i$.  For an  element $x\in R$, let $X$ be the orbit including $x$. Let $X\subseteq R_i$ for some $i\in I$, and $y\in Q\cap X$. There is an element $\phi=f_{a_1}\cdots f_{a_l}$ in $\mathrm{Inn}(R)$ for which $\phi(y)=x$. Note that for each $j\in I$, the elements of $R_j$ have the same action on the other elements of $R$. Therefore by the assumptions, we have $\phi\in\mathrm{Inn}(Q)$, and hence $x\in Q$. It follows that $R$ is a $G$-rack.
\end{proof}

In the next theorem, we characterize all racks with  abelian inner groups. Indeed, we show that such racks are exactly the same as the $P$-racks.
\begin{Theorem}
Let $R$ be a rack. Then $\mathrm{Inn}(R)$ is an abelian group if and only if $R$ is a $P$-rack. 
\end{Theorem}

\begin{proof}
By definition of a $P$-rack, it is obvious that the inner group of a $P$-rack is abelian. Conversely, assume that $R$ is a rack with  abelian inner group. We define the equivalence relation $\sim$ on $R$ as the following: for all $x,y\in R$, we have $x\sim y$ if and only if $f_x=f_y$. Now, we show that the partition given by these equivalence classes is a desired partition as in the definition of a $P$-rack. It is enough to show that for each $x,y\in R$, the elements $f_x(y)$, $f_x^{-1}(y)$ and $y$ belong to a same class, namely $f_{f_x(y)}=f_{f_x^{-1}(y)}=f_y$. The latter equalities hold, since we have $f_{f_x(y)}=f_xf_yf_x^{-1}=f_y$ and $f_{f_x^{-1}(y)}=f_x^{-1}f_yf_x=f_y$ by \cite[Lemma 2.1]{SK1}, and  $\mathrm{Inn}(R)$ is abelian by our assumption. 
\end{proof}

The following lemma plays an important role in the sequel.
\begin{Lemma}\label{Lemma}
Any rack is a $G$-rack if and only if its corresponding quandle is a $G$-rack.
\end{Lemma}

\begin{proof}
Let $(R,\tg)$ be a rack, and let $\overline{x}$ be the atom including $x$ for any $x\in R$.  For any subrack $Q$ of $R$, let $\overline{Q}=\{\overline{x}:x\in Q\}$. Then by \cite[Corollary 2.7]{SK1}, the map $Q\mapsto\overline{Q}$ defines a lattice isomorphism from $\mathcal{R}(R)$ to $\mathcal{R}(\overline{R})$ mapping the orbits of $R$ to the orbits of $\overline{R}$.

Now, let $R$ be a $G$-rack, and let $\{X_i\}_{i\in I}$ be the set of all orbits of $R$. Then $\{\overline{X_i}\}_{i\in I}$ is the set of all orbits of $\overline{R}$. We show that $\overline{R}$ is a $G$-rack. Let $\overline{Q}$ be a subrack of $\overline{R}$ such that $\overline{Q}\cap \overline{X_i}\neq\emptyset$ for any $i\in I$. Then $Q\cap X_i\neq\emptyset$ for any $i\in I$, since the map is one-to-one and $\overline{\emptyset}=\emptyset$. It follows that $Q=R$, since $R$ is a $G$-rack, and hence $\overline{Q}=\overline{R}$. Thus $\overline{R}$ is a $G$-rack as well. The converse follows similarly.
\end{proof}

An \textit{orthocomplemented} lattice  is a complemented lattice $L$ equipped with a function $\phi:L\to L$ such that $\phi^2=id$, $\phi(x)$ is a complement of $x$ in $L$, and $\phi(y)\le\phi(x)$ for all $x,y\in L$ with $x\le y$.

 In the next theorem, we characterize modular, relatively complemented and orthocomplemented  $G$-racks. 
 
 \begin{Theorem}\label{Prop}
 Let $R$ be a $G$-rack. Then the following statements are equivalent:
 \begin{enumerate}
 \item
 $\mathcal{R}(R)$ is a Boolean algebra.
 \item
 $\mathcal{R}(R)$ is modular.
 \item
 $\mathcal{R}(R)$ is relatively complemented.
 \item
 $\mathcal{R}(R)$ is orthocomplemented.
 \end{enumerate}
 \end{Theorem}
 
\begin{proof}
It is a well-known fact that (1) implies the other statements (see \cite[Section~6]{Gra}), and hence it is enough to show that each of  (2), (3) and (4) implies (1). By Lemma \ref{Lemma}, without loss of generality, we may assume that $R$ is a quandle. To see that $\mathcal{R}(R)$ is a Boolean algebra, it is enough to show that every orbit of $R$ is a singleton.  Let  $X$ be  an orbit of $R$ and $x\in X$. We can assume that $R$ has some orbits other than $X$, since otherwise, by definition of a $G$-rack, it follows that $R=\{x\}=X$, and hence there is nothing to prove. Note that if for a subrack $Q$ of $R$, we have $\ll X,Q\gg=R$, then $Q$ has a nonempty intersection with every orbit $Y\neq X$. Indeed, if  there is an orbit $Y\neq X$ with $Y\cap Q=\emptyset$, then $\ll X, Q\gg\subseteq R\backslash Y$, a contradiction. In particular, each complement of $X$ in $\mathcal{R}(R)$ has a nonempty intersection with every orbit of $R$ except $X$.  Now we consider the following cases:

First, suppose that $\mathcal{R}(R)$ is modular. If $X\neq\{x\}$, then $\mathcal{R}(R)$ has the sublattice $N_5$ shown in Figure \ref{f3}. Hence $\mathcal{R}(R)$ is not modular which is a contradiction. 

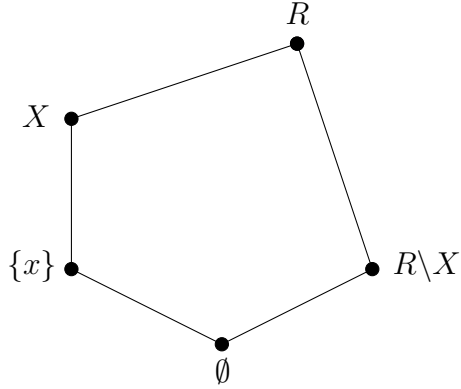
\begin{figure}
\definecolor{ududff}{rgb}{0.30196078431372547,0.30196078431372547,1.}
\begin{tikzpicture}[line cap=round,line join=round,>=triangle 45,x=1.0cm,y=1.0cm]
\clip(0,-2.55) rectangle (8,3);
\draw [line width=0.4pt] (4.,-2.)-- (2.,-1.);
\draw [line width=0.4pt] (4.,-2.)-- (6.,-1.);
\draw [line width=0.4pt] (2.,-1.)-- (2.,1.);
\draw [line width=0.4pt] (2.,1.)-- (5.,2.);
\draw [line width=0.4pt] (6.,-1.)-- (5.,2.);
\draw (3.76,-2.0354248486414646) node[anchor=north west] {$\emptyset$};
\draw (1,-0.6) node[anchor=north west] {$\{x\}$};
\draw (6.126431693679114,-0.6) node[anchor=north west] {$R\backslash X$};
\draw (1.2,1.33) node[anchor=north west] {$X$};
\draw (4.7,2.7) node[anchor=north west] {$R$};
\begin{scriptsize}
\draw [fill=black] (4.,-2.) circle (2.5pt);
\draw [fill=black] (2.,-1.) circle (2.5pt);
\draw [fill=black] (6.,-1.) circle (2.5pt);
\draw [fill=black] (2.,1.) circle (2.5pt);
\draw [fill=ududff] (5.,2.) circle (2.5pt);
\draw [fill=black] (5.,2.) circle (2.5pt);
\end{scriptsize}
\end{tikzpicture}
\caption{Nontrivial $G$-racks do not have the sublattice $N_5$. \label{f3}}
\end{figure}

Next, suppose that $\mathcal{R}(R)$ is relatively complemented. Consider the interval $[\{x\},R]$. Suppose that $T$ is  a complement of $X$ in the interval $[\{x\},R]$, and hence $T\cap X=\{x\}$ and $T$ has a nonempty intersection with every orbit of $R$. This implies that $T=R$, since $R$ is a $G$-rack. Now, it follows from $T\cap X=\{x\}$ that $X=\{x\}$. 

Finally, suppose that $\mathcal{R}(R)$ is orthocomplemented. Let $\phi:\mathcal{R}(R)\to\mathcal{R}(R)$ be a desired function as in the definition. Note that $\phi(X)$ has a nonempty intersection with every orbit of $R$ except $X$, since $\phi(X)$ is a complement of $X$ in $R$. We have $\phi(X)\subseteq\phi(\{x\})$, and hence $\phi(\{x\})$ has a nonempty intersection with every orbit of $R$ except $X$, as well. If $\phi(\{x\})\nsubseteq R\backslash X$, then $\phi(\{x\})$ has a nonempty intersection with every orbit of $R$, and hence $\phi(\{x\})=R$, since $R$ is a $G$-rack. But this is a contradiction, since  $\phi(\{x\})$ is a complement of $\{x\}$ in $\mathcal{R}(R)$. Therefore, we have $\phi(\{x\})\subseteq R\backslash X$, and hence $\phi(R\backslash X)\subseteq \phi^2(\{x\})=\{x\}$. This implies that $\phi(R\backslash X)=\{x\}$. Repeating the above arguments, for any $y\in X$, shows that $\phi(R\backslash X)=\{y\}=\{x\}$. Thus $X=\{x\}$.
 \end{proof}
 
 \begin{figure}[h]
\centering
\subfloat[The lattice $N_5$ in $\mathcal{R}(S_3)$.\label{NMM}]{
\definecolor{ududff}{rgb}{0.30196078431372547,0.30196078431372547,1.}
\begin{tikzpicture}[line cap=round,line join=round,>=triangle 45,x=1.0cm,y=1.0cm]
\clip(1.3,-2.55) rectangle (7,3);
\draw [line width=0.4pt] (4.,-2.)-- (2.,-1.);
\draw [line width=0.4pt] (4.,-2.)-- (6.,-1.);
\draw [line width=0.4pt] (2.,-1.)-- (2.,1.);
\draw [line width=0.4pt] (2.,1.)-- (5.,2.);
\draw [line width=0.4pt] (6.,-1.)-- (5.,2.);
\draw (3.75,-2.0354248486414646) node[anchor=north west] {$\emptyset$};
\draw (1.9,-0.5) node[anchor=north west] {\small{$\{(1\quad 2\quad 3)\}$}};
\draw (4.3,-0.5) node[anchor=north west] {\small{$\{(1\quad 2)\}$}};
\draw (1.2,1.5) node[anchor=north west] {$A_3$};
\draw (4.7,2.8) node[anchor=north west] {$S_3$};
\begin{scriptsize}
\draw [fill=black] (4.,-2.) circle (2.5pt);
\draw [fill=black] (2.,-1.) circle (2.5pt);
\draw [fill=black] (6.,-1.) circle (2.5pt);
\draw [fill=black] (2.,1.) circle (2.5pt);
\draw [fill=ududff] (5.,2.) circle (2.5pt);
\draw [fill=black] (5.,2.) circle (2.5pt);
\end{scriptsize}
\end{tikzpicture}
}
\qquad
\subfloat[The lattice $N_5$  in $\mathcal{R}(Q_8)$.\label{NMM1}]{
\definecolor{ududff}{rgb}{0.30196078431372547,0.30196078431372547,1.}
\begin{tikzpicture}[line cap=round,line join=round,>=triangle 45,x=1.0cm,y=1.0cm]
\clip(1.37,-2.55) rectangle (7,3);
\draw [line width=0.4pt] (4.,-2.)-- (2.,-1.);
\draw [line width=0.4pt] (4.,-2.)-- (6.,-1.);
\draw [line width=0.4pt] (2.,-1.)-- (2.,1.);
\draw [line width=0.4pt] (2.,1.)-- (5.,2.);
\draw [line width=0.4pt] (6.,-1.)-- (5.,2.);
\draw (3.76,-2.0354248486414646) node[anchor=north west] {$\emptyset$};
\draw (1.185,-0.6) node[anchor=north west] {$\{i\}$};
\draw (6.,-0.6) node[anchor=north west] {$\{j\}$};
\draw (1.195,1.75) node[anchor=north west] {$\{\pm i\}$};
\draw (4.1,2.8) node[anchor=north west] {$\{\pm i,\pm j\}$};
\begin{scriptsize}
\draw [fill=black] (4.,-2.) circle (2.5pt);
\draw [fill=black] (2.,-1.) circle (2.5pt);
\draw [fill=black] (6.,-1.) circle (2.5pt);
\draw [fill=black] (2.,1.) circle (2.5pt);
\draw [fill=ududff] (5.,2.) circle (2.5pt);
\draw [fill=black] (5.,2.) circle (2.5pt);
\end{scriptsize}
\end{tikzpicture}
}
\qquad
\subfloat[The lattice $N_5$  in $\mathcal{R}(D_8)$.\label{NMMM1}]{
\definecolor{ududff}{rgb}{0.30196078431372547,0.30196078431372547,1.}
\begin{tikzpicture}[line cap=round,line join=round,>=triangle 45,x=1.0cm,y=1.0cm]
\clip(1.25,-2.55) rectangle (7,3);
\draw [line width=0.4pt] (4.,-2.)-- (2.,-1.);
\draw [line width=0.4pt] (4.,-2.)-- (6.,-1.);
\draw [line width=0.4pt] (2.,-1.)-- (2.,1.);
\draw [line width=0.4pt] (2.,1.)-- (5.,2.);
\draw [line width=0.4pt] (6.,-1.)-- (5.,2.);
\draw (3.76,-2.0354248486414646) node[anchor=north west] {$\emptyset$};
\draw (1.05,-0.6) node[anchor=north west] {$\{\sigma\}$};
\draw (6.,-0.6) node[anchor=north west] {$\{\tau\}$};
\draw (1.195,1.84) node[anchor=north west] {$\{\sigma,\sigma^3\}$};
\draw (3.75,2.8) node[anchor=north west] {$\{\sigma,\tau,\sigma^3,\sigma^2\tau\}$};
\begin{scriptsize}
\draw [fill=black] (4.,-2.) circle (2.5pt);
\draw [fill=black] (2.,-1.) circle (2.5pt);
\draw [fill=black] (6.,-1.) circle (2.5pt);
\draw [fill=black] (2.,1.) circle (2.5pt);
\draw [fill=ududff] (5.,2.) circle (2.5pt);
\draw [fill=black] (5.,2.) circle (2.5pt);
\end{scriptsize}
\end{tikzpicture}
}
\caption{The lattices of subracks of $S_3$, $Q_8$ and $D_8$ are not modular.\label{f2}}
\end{figure}
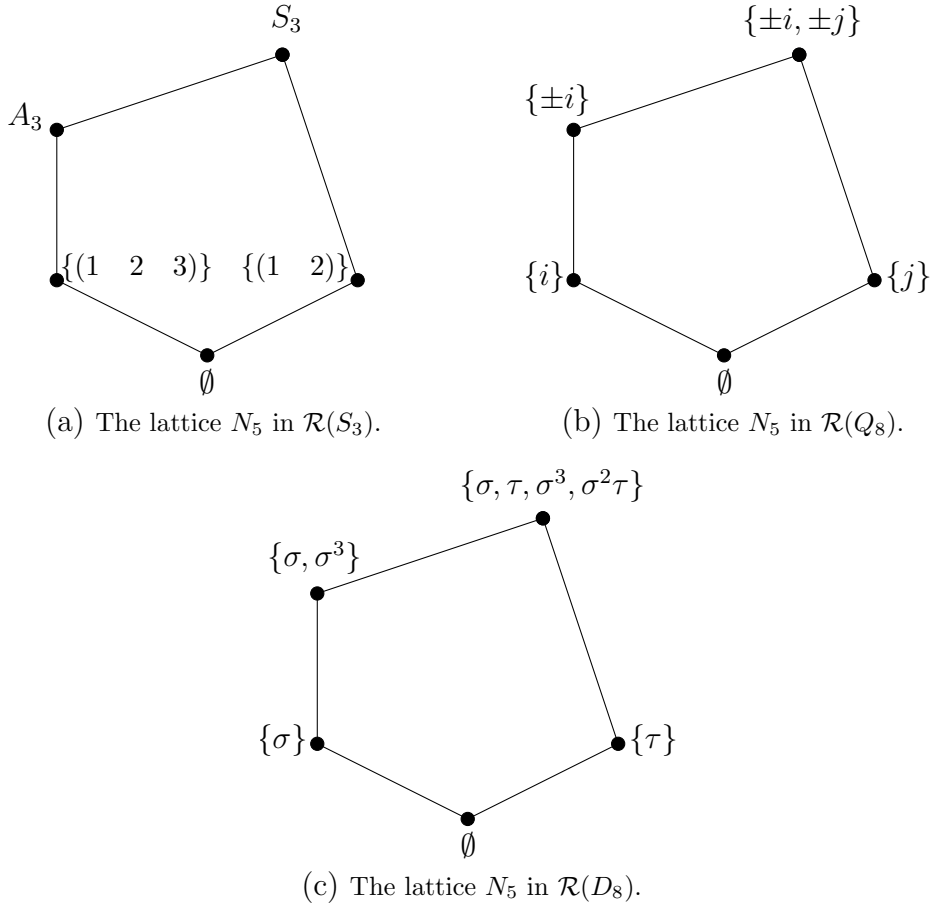
 As an immediate consequence of Theorem \ref{Prop}, the lattice of subracks of a finite group $G$ is modular, relatively complemented or orthocomplemented if and only if $\{x,y\}$  is a subrack of $G$ for all $x,y\in G$. Therefore, these equivalent conditions imply that  each two elements of $G$ commute with each other. Let $Q_8=\{\pm1, \pm i,\pm j,\pm k\}$ and $D_8=\left<\sigma,\tau:\sigma^4=1,\tau^2=1, \tau\sigma=\sigma^3\tau\right>$ be the quaternion group and dihedral group, respectively. In \cite{Wel},  all finite groups whose lattice of subracks is graded (i.e. all maximal chains of the lattice have the same length) were characterized.  Indeed, it was shown that the lattice of subracks of a finite group is graded if and only if the group is abelian, $S_3$, $D_8$ or $Q_8$. See Figure \ref{f2} to observe why $S_3$, $Q_8$ and $D_8$ have non-modular lattice of subracks. We would also like to remark that in  Theorem~\ref{Prop}, we  directly determined the groups (including infinite groups) with the modular lattice of subracks  without using the fact that modular lattices are graded. 

The \textit{order complex} of a poset $P$, denoted by $\Delta(P)$, is the set of all chains of $P$. Here, by the homotopy type of a poset $P$, we mean  the homotopy type of the order complex of $\overline{P}=P\backslash\{\hat{0},\hat{1}\}$. In the following, we determine the homotopy type of the lattice of subracks of a finite $G$-rack as a generalization of  \cite[Proposition~1.1]{Wel} for the lattice of subracks of a finite group. For this purpose, we need the well known nerve lemma (see \cite[Theorem~10.7]{Bjo}). The \textit{nerve} of a family $\mathcal{A}=\{A_i\}_{i\in I}$ of sets, denoted by $\mathcal{N}(\mathcal{A})$,  is a simplicial complex whose faces are the subsets $J$ of $I$ for which $\bigcap_{j\in J}A_j$ is nonempty. A \textit{covering} for a simplicial complex $\Delta$ is a collection $\{\Delta_i\}_{i\in I}$ of subcomplexes of $\Delta$ such that $\Delta=\bigcup_{i\in I}\Delta_i$.

\begin{Lemma}{(Nerve Lemma) }\label{Nerve}
Let $\Delta$ be a finite simplicial complex, and let $\mathcal{A}=\{\Delta_i\}_{i=1}^n$ be  a covering for $\Delta$. If for every $J\subseteq\{1,\ldots,n\}$, the intersection $\bigcap_{j\in J}\Delta_j$ is  $\{\emptyset\}$ or contractible, then $\mathcal{N}(\mathcal{A})$ and $\Delta$ have the same homotopy type.
\end{Lemma}

Now, we are ready to prove the following theorem.

\begin{Theorem}
Let $R$ be a finite $G$-rack which has $c$ orbits. Then $\mathcal{R}(R)$ has the homotopy type of  the $(c-2)$-dimensional sphere, $S^{c-2}$.
\end{Theorem}

\begin{proof}
For every subrack $Q$ of $R$, let 
\[\Delta_Q=\left\{\{\sigma_0,\ldots,\sigma_t\}\in\Delta(\overline{\mathcal{R}(R)}): \emptyset\subsetneq\sigma_0\subsetneq\cdots\subsetneq \sigma_t\subseteq Q\right\}.\]
In particular, $\Delta_{\emptyset}=\{\emptyset\}$. We have that $\mathcal{A}=\left\{\Delta_M: \text{$M$ is a maximal subrack of $R$}\right\}$ is a covering for $\Delta(\overline{\mathcal{R}(R)})$. Let $\{X_1,\ldots,X_c\}$ be the set of all orbits of $R$. Then $M$ is a maximal subrack of $R$ if and only if $M=R\backslash X_i$ for some $i=1,\ldots,c$, since $R$ is a $G$-rack. Let $M_i=R\backslash X_i$ for all $i=1,\ldots,c$, and let $I\subseteq\{1,\ldots,c\}$. Then it follows that $\bigcap_{i\in I}\Delta_{M_i}=\{\emptyset\}$ if and only if $I=\{1,\ldots,c\}$. This implies that the facets of $\mathcal{N}(\mathcal{A})$ are exactly the subsets of $\{1,\ldots,c\}$ with $c-1$ elements. Thus $\mathcal{N}(\mathcal{A})$ has the homotopy type of $S^{c-2}$. It is clear that for any nonempty  subrack $Q\neq R$,  the simplicial complex $\Delta_Q$ is a cone over $Q$, so that $\Delta_Q$ is contractible. On the other hand, for any $I\subseteq\{1,\cdots c\}$ we have $\bigcap_{i\in I}\Delta_{M_i}=\Delta_{\bigcap_{i\in I}M_i}$ which is $\{\emptyset\}$ or a cone over $\bigcap_{i\in I}M_i$, and hence it is $\{\emptyset\}$ or contractible. Now, it follows from Lemma \ref{Nerve} that the homotopy types of $\Delta(\overline{\mathcal{R}(R)})$ and $\mathcal{N}(\mathcal{A})$ are the same, and hence  $\Delta(\overline{\mathcal{R}(R)})$ has the homotopy type of $S^{c-2}$.
\end{proof}

Despite all finite groups are $G$-racks, in the next example we show that there are some infinite groups which are not $G$-racks. 

\begin{Example}
Let $H$ be the subgroup of all permutations on $\mathbb{N}$ fixing all elements  except a finite number of elements. One could consider $H$ as $\bigcup_{n\in\mathbb{N}}S_n$ where $S_n$ and its natural embedding in $S_{\mathbb{N}}$ are identified for all $n\in\mathbb{N}$. The orbits of $H$ are exactly its conjugacy classes. Hence any orbit of $H$ including an element $x$ also contains all elements of $H$ which have the same cycle structure as $x$. We show that $H$ is not a $G$-rack. Note that the set of all orbits of $H$ is countable. Let $\{C_i\}_{i\in\mathbb{N}}$ be the set of all orbits of $H$. Then we can choose inductively pairwise disjoint permutations $g_i\in C_i$. Let $Q=\{g_i\}_{i\in I}$. Then any  two elements of $Q$ commute with each other, and hence $Q$ is a subrack of $H$. Therefore $Q$ is a proper subrack of $H$ including an element of every orbit of $H$ which implies that $H$ is not a $G$-rack.
\end{Example}

\section{Complementedness of The Lattices of Subracks of Infinite Racks}\label{S}
In this section, we show that the lattice of subracks of an infinite rack is not necessarily complemented which answers  Question~\ref{Que}. To do this, we show that the lattice of subracks of $\mathbb{Q}$, as a dihedral rack, is not complemented. This example also shows that a rack may not have a maximal subrack. Indeed in Section~\ref{sec1} to prove that the lattice of subracks of every finite rack is complemented, we  used the fact that every subrack is contained in a maximal subrack. We also provide examples of infinite racks whose lattices of subracks are complemented. Although all trivial racks have  complemented lattices of subracks, there are some non-trivial infinite racks whose lattices of subracks are complemented.

\medskip 
First, in the following lemma, we characterize all subracks of $\mathbb{Q}$ as a dihedral rack. We use this characterization to show that the dihedral rack $\mathbb{Q}$ is not complemented.

\begin{Lemma}\label{Lemma1}
The nonempty subracks of $\mathbb{Q}$ (as a dihedral rack) are exactly the cosets of the subgroups of $\mathbb{Q}$.
\end{Lemma}

\begin{proof}
Let $Q$ be a nonempty subrack of $\mathbb{Q}$. First, assume that $0\in Q$. Then for any integer $s$ and any $z\in Q$, one has $sz\in Q$. Indeed, if $s\ge 0$, then $sz\in Q$ and $-sz\in Q$ by induction on $s$, since $(sz)\tg\left((s-1)z\right)=(s+1)z$ and $0\in Q$. Now, we show that $Q$ is a subgroup of $\mathbb{Q}$. Let $x,y\in Q$ where $x=a/b$ and $y=c/d$  for some integers $a,b,c,d$. If $x=0$ or $y=0$, then obviously $x-y\in Q$. Let $e=\mathrm{gcd}(ad,bc)$. Then either $e=\mathrm{gcd}(2ad,bc)$ or $e=\mathrm{gcd}(ad,2bc)$. Without loss of generality, assume that $e=\mathrm{gcd}(2ad,bc)$. Thus, there are integers $s,t$ for which $2sad-tbc=e$. So, it follows from $(sx)\tg(ty)\in Q$ that $e/bd\in Q$, and hence $x-y$, which is a multiple of $e/bd$, belongs to $Q$.\\
Next, assume that $Q$ is an arbitrary nonempty subrack of $\mathbb{Q}$ and $x\in Q$. One could observe that $Q-x$ is a subrack of $\mathbb{Q}$ including $0$, and hence by the above argument $Q-x$ is a subgroup of $\mathbb{Q}$. Therefore, $Q$ is a coset of a subgroup of $\mathbb{Q}$.
\end{proof}

Now, we are ready to show that there are infinite racks whose lattices of subracks are not complemented. 

\begin{Theorem}
The lattice of subracks of $\mathbb{Q}$ (as a dihedral rack) is not complemented.
\end{Theorem}

\begin{proof}
Suppose that $Q$ is a nonempty subrack of $\mathbb{Q}$. Then $Q=H+x$ where $H$ is a subgroup of $\mathbb{Q}$ and $x\in Q$ by Lemma \ref{Lemma1}. Now let $L=\ll Q,0\gg$ which is a subgroup of $\mathbb{Q}$ by Lemma \ref{Lemma1}. We show that $L=H+\left<x\right>$. Since $H+\left<x\right>$ contains $Q\cup\{0\}$, we have $L\subseteq H+\left<x\right>$. On the other hand,  for any $h\in H$ and any integer $s$, we have $h+sx=(h+x)+(s-1)x\in L$, since $L$ is a subgroup of $\mathbb{Q}$. Therefore $H+\left<x\right>\subseteq L$.

Now, we show that the subrack $\{0\}$ does not have any complement in $\mathcal{R}(\mathbb{Q})$. On contrast, assume that $Q=H+x$ is a complement of $\{0\}$ in $\mathcal{R}(\mathbb{Q})$, where $H$ is a subgroup of $\mathbb{Q}$ and $x\in Q$. It follows that $H+\left<x\right>=\mathbb{Q}$, and hence $\mathbb{Q}/H$ and $\left<x\right>/\left(\left<x\right>\cap H\right)$ are isomorphic groups. Since $0\notin Q$ it follows that  $x\neq 0$. Let $x=r/s$.  Then $\{0\}\neq\left<r\right>\subseteq\left<x\right>$. Moreover, $H\neq\{0\}$ implies that there exists a nonzero element $y=t/u\in H$, and hence $\{0\}\neq\left<t\right>\subseteq H$. Therefore,  $\left<rt\right>\subseteq H\cap\left<x\right>$, and hence $\left<x\right>/\left(\left<x\right>\cap H\right)$ is a finite cyclic group. This implies that $H$ is a proper subgroup of $\mathbb{Q}$ with finite index which is a contradiction.
\end{proof}

In the following example, we show that the lattice of subracks of every $G$-rack is complemented. 

\begin{Example}
Let $R$ be a $G$-rack, and let $Q$ be a subrack of $R$. Assume that $L$ is the union of all orbits of $R$ whose intersections with $Q$ are all empty. Then $L$ and $Q$ are disjoint and $\ll Q, L\gg$ has a nonempty intersection with any orbit of $R$, and hence $\ll Q, L\gg=R$. It follows that $L$ is a complement of $Q$ in $\mathcal{R}(R)$. As a consequence, the lattice of subracks of every $P$-rack is complemented as well.
\end{Example}

In the next example, we provide an infinite quandle whose lattice of subracks is complemented. 

\begin{Example}
Let $R$ be the quandle of all transpositions in $S_{\mathbb{N}}$ equipped with the conjugation operation. We show that $\mathcal{R}(R)$ is complemented. Note that for any subrack $Q$ of $R$ and any transposition $(a\;\;b)$ of $R$, we have 
\[ \ll Q, (a \;\; b)\gg=Q\cup\{(a\;\; x): \exists x\in\mathbb{N}\;(x\;\;b)\in Q\}\cup\{(b\;\; x): \exists x\in\mathbb{N}\;(x\;\;a)\in Q\},\]
by \cite[Lemma 2.2]{SK1}.
Now, let $Q$ be a subrack of $R$ and $\Gamma$ be the set of all subracks $T$ of $R$ with $T\cap Q=\emptyset$. One could easily observe that $\Gamma$ is nonempty and any chain in $\Gamma$ has an upper bound in $\Gamma$. Thus $\Gamma$ has a maximal element $Q'$ by Zorn's Lemma. We show that $Q'$ is a complement of $Q$ in $\mathcal{R}(R)$. Suppose on contrast that $\ll Q,Q'\gg\neq R$, and hence there exists a transposition $(a\;\; b)\in R\backslash\ll Q,Q'\gg$. By maximality of $Q'$ in $\Gamma$, there exists an element $(x\;\;y)\in\;\ll Q', (a\;\;b)\gg\cap\;Q$. Since $(x\;\;y)\notin Q'$ and $(x\;\;y)\in\;\ll Q', (a\;\;b)\gg$, we have $\{x,y\}\cap\{a,b\}\neq\emptyset$. Hence  we may consider $(x\;\;y)=(a\;\;z)$ and $(b\;\;z)\in Q'$ for some $z\in\{x,y\}$. This implies that $(a\;\;b)=(a\;\;z)(b\;\;z)(a\;\;z)\in\;\ll Q,Q'\gg$, a contradiction. Therefore, $Q'$ is a complement of $Q$ in $\mathcal{R}(R)$.
\end{Example}

For any rack $R$, by considering $\mathrm{Inn}(R)$ with the conjugation operation, we have the homomorphism of racks  $\phi:R\to\mathrm{Inn}(R)$ defined by $\phi(r)=f_r$ for any $r\in R$. Indeed, we have $\phi(r\tg s)=f_{r\tg s}=f_{f_r(s)}=f_rf_sf_r^{-1}=f_r\tg f_s$. Now, by considering the dihedral rack $\mathbb{Q}$, the homomorphism $\phi$ is injective, and hence $\phi(\mathbb{Q})$ is a subrack of $\mathrm{Inn}(Q)$ which is isomorphic to the rack $\mathbb{Q}$. This implies that infinite groups might have some subracks which are not complemented and do not have any maximal subracks.

%Let $(R,\tg)$ be a rack. It is easily seen that the intersection and the subrack generated by  every family of subracks of $R$ are subracks as well, and hence  $\mathcal{R}(R)$ is a complete lattice. 

In the next theorem, we provide some equivalent statements concerning the structure of the lattice of subracks of a rack.

\begin{Theorem} \label{equi}
Let $R$ be a rack. Then the following statements are equivalent:
\begin{enumerate}
\item
$\mathcal{R}(R)$ is a Boolean algebra. 
\item
$\mathcal{R}(R)$ is distributive.
\item
$\mathcal{R}(R)$ is pseudocomplemented.
\item
$\mathcal{R}(R)$ is uniquely complemented.
\end{enumerate}
\end{Theorem}

\begin{proof}
It was shown in \cite[Theorem 2.8]{SK1} that the statements (1) and (2) are equivalent. Now, we show that the other statements are equivalent to the statement (1). By \cite[Corollary 2.7]{SK1}, we have that  the lattice of subracks of every rack is isomorphic to the one for its corresponding quandle, and hence here it is enough to consider quandles.

First, assume that $\mathcal{R}(R)$ is pseudocomplemented. It is enough to show that $f_x(y)=y$ for all $x,y\in R$. Assume on contrast  that there exist $x,y\in R$ with $f_x(y)\neq y$.  For every $z\in R\backslash\{y\}$, it follows from $\{y\}\wedge\{z\}=\hat{0}$ that $z\in\{y\}^*$, and hence $\{y\}^*=R\backslash\{y\}$. Since $f_x(y)\neq y$, we have $f_x(y)\in \{y\}^*$. Hence $y=f_x^{-1}(f_x(y))\in \{y\}^*$, because $x\in \{y\}^*$, which is a contradiction. Therefore $\mathcal{R}(R)$ is a Boolean algebra. Conversely, it is well-known that every Boolean algebra is pseudocomplemented.

Next, suppose that $\mathcal{R}(R)$ is uniquely complemented.  Assume on contrast that $f_x(y)\neq y$ for some $x,y\in R$. Consider $Q$ as the unique complement of $\{y\}$ in $\mathcal{R}(R)$, and hence $\ll y,Q\gg=R$ and $y\notin Q$. One could observe that  
\[\ll Q, f_q(y)\gg=\ll f_y(Q),y\gg=\ll y,Q\gg=R,\]
for any $q\in Q$. It follows from $y\notin Q$ that  $y\notin f_y(Q)$ and $f_q(y)\notin Q$ for any $q\in Q$. Therefore, by uniquely complementedness of $\mathcal{R}(R)$, we have $f_y(Q)=Q$ and $f_q(y)=y$ for any $q\in Q$. This implies that $R=\ll y,Q\gg=Q\cup\{y\}$, and hence $Q=R\backslash\{y\}$. Then, similar to the proof of equivalency of (1) and (3), we get a contradiction, and hence $\mathcal{R}(R)$ is a Boolean algebra (see also \cite[Theorems 1,2]{Bir}). The converse  follows from the well-known fact that every Boolean algebra is a uniquely complemented lattice. 
\end{proof}
\section*{Acknowledgements}
The authors would like to thank Sara Saeedi Madani for her useful comments. The authors would also like to thank the institute for research in fundamental sciences (IPM) for financial support. The research of the second author was in part supported by a grant from IPM (No. 96050212).

\end{document}